\documentclass[12pt]{amsart}
\usepackage{epsfig}
\usepackage{amscd}
\usepackage[mathscr]{eucal}
\usepackage{amssymb}
\usepackage{amsxtra}
\usepackage{amsmath}
\usepackage[all]{xy}

\newtheorem{thm}{Theorem}

\newtheorem{lem}{Lemma}

\newtheorem{clm}{Claim}

\newtheorem{cor}{Corollary}

\begin{document}

\title[growth Gap Vs. Smoothness]{GROWTH GAP VERSUS SMOOTHNESS FOR DIFFEOMORPHISMS OF THE INTERVAL}
\date{July 20, 2008.}
\author[Lev Buhovski and Roman Muraviev]{Lev Buhovski and Roman Muraviev\\
School of Mathematical Sciences\\
Tel-Aviv University}
\address{School of Mathematical Sciences, Tel-Aviv
  University, Ramat-Aviv, Tel-Aviv 69978, Israel.}
\email{levbuh@post.tau.ac.il, moraviev@post.tau.ac.il. }



\begin{abstract}

Given a diffeomorphism of the interval, consider the uniform norm
of the derivative of its $n$-th iteration. We get a sequence of real
numbers called the growth sequence. Its asymptotic behavior is an
invariant which naturally appears both in smooth dynamics and in
geometry of the diffeomorphisms groups. We find sharp estimates
for the growth sequence of a given diffeomorphism in terms of the
modulus of continuity of its derivative. These estimates extend
previous results of Polterovich-Sodin and Borichev.

\end{abstract}

\maketitle

\section{Introduction and main results} \label{S:intro}

Denote by Diff$_{0}[0,1]$ the group of all $C^1$-smooth
diffeomorphisms  of the interval $[0,1]$ fixing the end points 0
and 1. For any $f \in$Diff$_{0}[0,1]$, we define the growth
sequence of $f$ by $$ \Gamma_{n}(f)= \max \{
\|(f^n)'(x)\|_{\infty}, \|(f^{-n})'(x)\|_{\infty} \},$$ for all $n
\in \mathbb{N},$ where $\|.\|_{\infty}$ stands for the uniform
norm.

We say that a subgroup $G \subseteq $Diff$_{0}[0,1]$ admits a
\textit{growth gap} if there exists a sequence of positive numbers
$\gamma_n(G)$ that grows sub-exponentially to $+\infty$, such that
for any $f \in G$, either $\Gamma_{n}(f)$ tends exponentially to
$+ \infty$, or $\Gamma_{n}(f) \leq C(f) \cdot
\gamma_{n}(G),$ for all $n \in \mathbb{N}.$

From a viewpoint of dynamics, growth sequence of an element reflects
how the length changes asymptotically under iterations. At the same
time, geometrically, growth sequence indicates how an element is
distorted with respect to the multiplicative norm. In $[DG],$
D'Ambra and Gromov suggested to study growth sequences of various
classes of diffeomorphisms.

The growth sequence is always submultiplicative:
$$\Gamma_{m+n}(f) \leq \Gamma_{m}(f) \cdot \Gamma_{n}(f),$$
for all $m,n \in \mathbb{N}$. Therefore, the limit $$ \gamma(f) =
\lim_{n \rightarrow \infty} \sqrt[n]{\Gamma_{n}(f)}$$ always
exists. Using standard arguments of ergodic theory, one can check
that
$$\gamma(f)=1 \ \ {\rm if \ and \ only \ if} \ \  f'(\xi)=1 \ {\rm for \ every} \ \xi \in {\rm Fix}(f),$$
(see $[PS]$, page 199). The following theorem shows that the whole
group Diff$_{0}[0,1]$ does not admit a growth gap (see $[B]$),

\begin{thm} \label{T:Boricev1}
Given any monotone decreasing sequence of positive numbers $\{
\alpha_{n} \}_{n=1}^{\infty}$  tending to $0$, there exists $f
\in$Diff$_{0}[0,1]$ such that $Fix(f)=\{0,1\}$, $\gamma(f)=1$ and
\[
\Gamma_{n}(f) \geq e^{\alpha_{n} \cdot n}
\]
for all $n \in \mathbb{N}.$
\end{thm}

As it is shown in Theorem ~\ref{T:Boricev1}, weakening  of
smoothness assumptions leaves more room for exponential growth,
i.e., the growth sequence $\Gamma_{n}(f)$ becomes bigger, or in
other words, "the growth gap" is smaller. Therefore, smaller
subgroups of Diff$_{0}[0,1]$ should be considered in order to
discover a growth gap. In $[{\rm PS}]$ a growth gap was found for
the subgroup of $C^2$-diffeomorphisms of Diff$_{0}[0,1]$. Namely,

\begin{thm} \label{T:Polterovich Sodin}
 Let $f \in$Diff$_{0}[0,1]$ be a $C^2$-diffeomorphism with $\gamma(f)=1$.
Then
\[
\Gamma_{n}(f) \leq C(f) \cdot n^2 ,
\]
for all $n \in \mathbb{N}.$
\end{thm}

This result leads to a natural question on the growth gap for
subgroups of Diff$_{0}[0,1]$ with intermediate smoothness rate
between $C^1$ and $C^2$. A partial answer is provided in $[{\rm
B}]$. To introduce this result, we consider the following subgroup
of Diff$_{0}[0,1]$ which is associated with the H\"{o}lder
condition, H$_{\alpha}[0,1]= \{f \in$ Diff$_{0}[0,1] : |f'(x)-f'(y)|
\leq C(f) \cdot |x-y|^{\alpha} \}$, for $0<\alpha<1$.

\begin{thm} \label{T:Boricev2}
If $f \in H_{\alpha}[0,1]$ with $\gamma(f)=1$, then \\
$$\log\Gamma_{n}(f) \leq C(f, \alpha) \cdot n^{1-\alpha},$$
for all $n \in \mathbb{N}.$
\end{thm}

 In the present work we obtain
a growth gap for the following intermediate subgroups of
diffeomorphisms: \\

\textit{Case (a)}: Subgroups between $C^2[0,1] \bigcap$Diff$_{0}[0,1]$
and $\cap_{0<\alpha<1}H_{\alpha}[0,1].$

\textit{Case (b)}: Subgroups between Diff$_{0}[0,1]$ and
 $\cup_{0<\alpha<1}H_{\alpha}[0,1].$ \\

To describe subgroups of smoothness between $C^1$ and $C^2$, we
use the terminology of moduli of continuity, i.e., non-decreasing
continuous functions $\omega:[0,1] \rightarrow \mathbb{R}$
satisfying $\omega(0)=0$ and $\omega(\delta_1+\delta_2) \leq
\omega(\delta_1)+\omega(\delta_2).$ Given a modulus of continuity
$\omega:[0,1]\rightarrow \mathbb{R}_{+}$, we consider the subgroup
$${\rm Diff}^{\omega}_0 [0, 1]=\{ f\in {\rm Diff}_{0}[0,1]:
\omega_{f'}(\delta) \leq C(f) \cdot \omega(\delta) \},$$ where
$\omega_{f'}(\delta)=\max_{|x-y| \leq \delta }|f'(x)-f'(y)|.$
\\

It is not hard to check that ${\rm Diff}^{\omega}_0 [0, 1]$ is a
non-empty subgroup. Indeed, the identity map is an element of
${\rm Diff}^{\omega}_0 [0, 1]$. Furthermore, for any two $f, g \in
{\rm
Diff}^{\omega}_0 [0, 1],$ \\
$$ |(f \circ g)'(x)-(f \circ g)'(y)| $$
$$\leq |f'(g(x))g'(x) - f'(g(x))g'(y) | + | f'(g(x))g'(y) - f'(g(y))g'(y) |$$
$$\leq A(f) \cdot |g'(x)-g'(y)| + B(g) \cdot |f'(g(x)) - f'(g(y))| \leq C(f,g) \cdot \omega(|x-y|).$$

$$|(f^{-1})'(x)- (f^{-1})'(y)| = |\frac{ f'(f^{-1}(x)) - f'(f^{-1}(y)) }{f'(f^{-1}(x)) \cdot f'(f^{-1}(y))}| $$ $$\leq |\frac{ f'(f^{-1}(x)) -
f'(f^{-1}(y))}{a(f)^2}| \leq B(f) \cdot
\omega(|f^{-1}(x)-f^{-1}(y)|) $$ $$\leq C(f) \cdot \omega(|x-y|).$$
\\ Our first result generalizes Theorems \ref{T:Polterovich Sodin}
and \ref{T:Boricev2} and provides a growth gap for \textit{case (a)}.

\begin{thm} \label{T:generalization}
 Let $\omega(x):[0,1] \rightarrow \mathbb{R}_{+}$ be a
strictly increasing modulus
of continuity. Then, for each ${f \in {\rm Diff}}^{\omega}_0 [0, 1],$ such that $\gamma(f)=1$, we have \\
\[
(*) \ \ \  \log{\Gamma_{n}(f)} \leq
\log{\frac{n}{\omega^{-1}(\frac{2}{n})}} +  C(f) n
\omega(\frac{1}{n}).
\]
Here we denote by $\omega^{-1}$ the inverse function to $\omega.$
\end{thm}
One can substitute $\omega(\delta)=\delta$ and
$\omega(\delta)=\delta^{\alpha}$ into Theorem
~\ref{T:generalization} for achieving Theorems ~\ref{T:Polterovich
Sodin} and ~\ref{T:Boricev2}. In the following corollary, we
consider two toy models related to \textit{case (a)} in order to test how
Theorem 4 provides a growth gap. In the case when the modulus of
continuity $\omega(\delta)$ is close to the identity, the second
term on the right hand side of $(*)$ can be absorbed into the
first one. Namely,

\begin{cor} \label{C:Partial cases}
 (1) If $$\limsup_{x \rightarrow 0}
\frac{\omega(x)}{x \cdot \log{\frac{e}{x}}} < +\infty,$$ then

$$ \log{\Gamma_{n}(f)} \leq  C(f, \omega) \cdot \log{\frac{n}{\omega^{-1}(\frac{2}{n})}}. $$

(2) If
$$\lim_{x \rightarrow 0}
\frac{\omega(x)}{x \cdot \log{\frac{e}{x}}} =0,$$ then

$$ \log{\Gamma_{n}(f)} \leq  (1+o(1)) \cdot
\log{\frac{n}{\omega^{-1}(\frac{2}{n})}}.
$$ \\
\end{cor}

The proofs easily follow by substituting the relevant assumptions
into Theorem~\ref{T:generalization}.

The drawback of Theorem~\ref{T:generalization} is that it does not
provide a growth gap for \textit{case (b)}. For instance, if we consider a
diffeomorphism $f(x)$ from \textit{case (b)} with $\omega_{f'}(\delta) \leq
\frac{1}{\log{\frac{e}{\delta}}}$, then an attempt to apply
Theorem~\ref{T:generalization} for this diffeomorphism yields only
a trivial estimate
$$\log{\Gamma_{n}(f)} \leq C(f) \cdot n.$$
Our second theorem mends this disadvantage. It shows that in \textit{case
(b)} (under additional regularity assumption imposed on $\omega$)
one can discard the first term on the right hand side of $(*):$

\begin{thm} \label{T:low smoothness}
 Let $\omega(x):[0,1] \rightarrow
\mathbb{R}_{+}$ be a modulus of continuity such that for some
$0<\alpha<1$, $\frac{\omega(x)}{x^{\alpha}}$ is a decreasing
function on $(0,a(\alpha))$,  where $0<a(\alpha)<1$. Then for $f
\in$ \rm Diff$^{\omega}_{0}[0,1]
 ,$ such that $\gamma(f)=1$, we have \\
\[
\log{\Gamma_{n}(f)} \leq C(f) \cdot n \omega(\frac{1}{n}).
\]
\end{thm}

 The next set of theorems present a sufficient sharpness for
the estimates of the bounds in Theorems \ref{T:generalization} and
\ref{T:low smoothness} respectively.

\begin{thm} \label{T:sharpness1}
 Suppose that for each $0<\alpha<1$ there exists $0<a(\alpha)<1$ such that the function $\tfrac{\omega (x)}{x^\alpha}$
increases for all $x \in [0,a(\alpha)]$ and suppose that
$$\lim_{x \rightarrow 0}\tfrac{\omega(x)}{x \cdot
\log(\frac{e}{x})}=0.$$ Then, there exists a diffeomorphism $f\in
{\rm Diff}^{\omega}_0 [0, 1]$ with $\gamma(f)=1$ such that for any
$\varepsilon>0,$ \\
\[
\log \Gamma_n(f) \ge (1-\varepsilon) \cdot
\log\frac{n}{\omega^{-1}(\frac{c(f)}{n})}\  , \qquad n\to\infty\,.
\]
\end{thm}

\begin{thm} \label{T:sharpness2}
 Suppose that the modulus of continuity $\omega$
satisfies assumptions of Theorem~\ref{T:low smoothness}. Then
there exists a diffeomorphism $f\in {\rm Diff}^{\omega}_0 [0, 1]$
with $\gamma(f)=1, $ such that for each $\varepsilon > 0$,
\[
\log \Gamma_n (f) \ge c(\varepsilon) n^{1-\varepsilon} \omega
\left( \frac1{n} \right)\,, \qquad n\to\infty\,.
\]
\end{thm}

The proofs of Theorems 4-7 use ideas and techniques introduced in
${\rm [L, chapter II]}$ and especially in ${\rm [B]}.$

\bigskip

\section{Growth gap: Proofs of theorems 4 and 5} \label{S:proofs1}

The following lemma (see ${\rm [EF, Dz]}$) states that every
modulus of continuity admits an equivalent concave modulus of
continuity:

\begin{lem} \label{L:3}
For any modulus of continuity $\omega$ there exists a concave
modulus of continuity $\omega^*$ such that $\omega \le \omega^*
\le 2\omega$ everywhere on $[0, 1]$.
\end{lem}

Due to this lemma, we assume in the proofs of Theorems 4 and 5
that $\omega$ is a concave modulus of continuity.

\begin{proof}[Proof of Theorem~\ref{T:generalization}]

First, we will introduce several notations and definitions:
$\phi(x):= f(x) - x;$ $x_n = f(x_{n-1})$; $A=\max_{x \in
[0,1]}f'(x)$, $a=\min_{x \in [0,1]}f'(x)$. Choose a sufficiently
small $ \varepsilon
>0,$ such that we will have
$\omega(\varepsilon) < 1.$ WLOG, we assume that $\phi(x)$ is
positive. Consider a function $x \mapsto x \cdot \omega(x)$ which
maps $[0,\varepsilon]$ on $[0,\varepsilon \cdot
\omega(\varepsilon)],$ and denote by $\Omega(x):[0,\varepsilon \cdot
\omega(\varepsilon)] \rightarrow [0,\varepsilon]$ its inverse. Now
pick a positive $\delta < \varepsilon $, so that the following
requirement will be satisfied:

For all $x \in [0, \delta]$ we have $ \phi(x) \in [0,\varepsilon
\cdot \omega(\varepsilon)]$ and
$$J_x:=[x,f(x)] \subseteq I_x := [x - \Omega(\phi(x)),x + \Omega(\phi(x))]
\subseteq [0, \varepsilon].$$ \\
Let us explain why it is possible. It is obviously possible to
require that $\phi(x) \in [0,\varepsilon \cdot
\omega(\varepsilon)]$ and $x + \Omega(\phi(x)) \leq \varepsilon$
for all $x \in [0,\delta],$ due to continuity. The inequality
$$ 0 \leq x - \Omega(\phi(x)) $$
is equivalent to that
$$ \phi(x) \leq x \cdot \omega(x)$$
which is satisfied for all $x \in [0,\delta],$ since $|\phi'(x)|
\leq \omega(x).$

We will present now a sequence of technical claims, which will be
used later in the proof of Theorem~\ref{T:generalization}.

\begin{clm} \label{Cl:1}
(a) For any $x \in [0,1]$ and $y \in [x,f(x)]$,
$$ \frac{1}{A} \leq \frac{\phi(x)}{\phi(y)} \leq \frac{1}{a}.$$

(b)  For any $ x_1 \in [0,1]$ and $n \in \mathbb{N},$
 $$ \frac{1}{A} \cdot n \leq
\int_{x_1}^{x_{n+1}}\frac{dt}{\phi(t)} \leq \frac{1}{a} \cdot n.$$
\end{clm}

\begin{proof}[Proof of Claim~\ref{Cl:1}]
 For any $y \in [x,f(x)]$, there
exists $ 0 \leq \theta \leq 1 $ such that $y=x+\theta\cdot\phi(x)$
and $0 \leq \theta_{1} \leq 1$, such that
$$\frac{\phi(y)}{\phi(x)} = \frac{\phi(x)+\theta\phi(x)\phi'(x+\theta_{1}\theta\phi(x))}{\phi(x)} \leq 1 + \max_{x \in [0,1]} {\phi'(x)} \leq
A.$$
 In the same way,
$$\frac{\phi(x)}{\phi(y)} = \frac{\phi(x)}{\phi(x)+\theta\phi(x)\phi'(x+\theta_{1}\theta\phi(x))} \leq \frac{1}{1+\phi'(x+\theta_{1}\theta\phi(x))}
\leq \frac{1}{a}.$$ Therefore, for all $k \in \mathbb{N}:$
$$\frac{1}{A} \leq \int_{x_k}^{x_{k+1}}\frac{dt}{\phi(t)} \leq
\frac{1}{a} .$$ By summing the integrals we obtain the desirable
inequality.
\end{proof}

\begin{clm} \label{Cl:2}
 For all $x \in [0,\delta]$ and $y \in I_x,$ we
have

$$(a) \ \ \ |\phi'(y)| \leq 3 \omega(\Omega(\phi(x))) = 3
\frac{\phi(x)}{\Omega(\phi(x))}.$$

$$ (b)  \ \ \ \frac{\phi(y)}{\phi(x)} \leq 2.5 . $$ \\ \\
\end{clm}

\textbf{Remark:} In particular, we obtain that for all $x \in
[0,\delta],$ $|\phi'(x)| \leq 3  \omega(\Omega(\phi(x)))$.

\begin{proof}[Proof of Claim~\ref{Cl:2}]
(a) Suppose that there exists $y_0 \in I_x$ such that $\phi'(y_0)
> 3 \cdot \omega(\Omega(\phi(x)))$. Note that the following
inequalities are satisfied for all $y\in I_x$:
$$\phi'(y)\geq\phi'(y_0)-\omega(|y-y_0|)>3 \cdot \omega(\Omega(\phi(x)))-\omega(2\Omega(\phi(x)))$$ $$\geq (3-2) \cdot \omega(\Omega(\phi(x)))
= \omega(\Omega(\phi(x))).$$ Therefore,
$$\phi(x)-\phi(x-\Omega(\phi(x))) =
\int_{x-\Omega(\phi(x))}^{x}\phi'(t)dt > \omega(\Omega(\phi(x)))
\cdot \Omega(\phi(x))=\phi(x). $$ It follows that
$\phi(x-\Omega(\phi(x)))<0$, what contradicts our assumptions.
Now, assume that there exists a point $y_0 \in I_x$ such that
$\phi'(y_0)<-3 \cdot \omega(\Omega(\phi(x)))$. Then for all $y \in
I_x,$
$$\phi'(y)\leq \phi'(y_0)+\omega(|y-y_0|) < -3 \cdot \omega(\Omega(\phi(x)))+\omega(2\Omega(\phi(x)))$$ $$\leq (-3+2)\omega(\Omega(\phi(x))) \leq
-\omega(\Omega(\phi(x))).$$ Therefore, $$\phi(x+\Omega\phi(x)) -
\phi(x) =
\int_{x}^{x+\Omega(\phi(x))}\phi'(t)dt<-\omega(\Omega(\phi(x)))
\cdot \Omega(\phi(x))=-\phi(x),$$ whence $
\phi(x+\Omega\phi(x))<0$, this is a contradiction.  \\
(b) Using (a) we obtain for some $0<\theta_{1}<1$,
$$\frac{\phi(y)}{\phi(x)}=\frac{\phi(x)+(y-x)\phi'(x+\theta_{1}(y-x))}{\phi(x)}\leq 1 + \max_{y\in I_{x}}|\phi'(y)| \cdot \frac{|I_x|}{2 \phi(x)}\leq$$
$$\leq 1 + \frac{3}{2} \cdot \omega(\Omega(\phi(x)))\frac{\Omega(\phi(x))}{\phi(x)}= 2.5 .$$
\end{proof}

\begin{clm} \label{Cl:3}
 Let $z \in [0,\delta]$ and $n \in \mathbb{N}$ be such that $$n \geq
c(f) \cdot \frac{\Omega(z)}{z}.$$ Then,
$$\frac{1}{z} \leq C(f) \cdot
\frac{n}{\omega^{-1}(\frac{1}{n})}.$$
\end{clm}

\begin{proof}[Proof of Claim~\ref{Cl:3}]
 Denote $s=\Omega(z)$, and notice
that $s \cdot \omega(s) = z.$ Thus,
$$\omega(s)=\frac{z}{\Omega(z)} \geq \frac{c(f)}{n}$$
$$s \geq \omega^{-1}(\frac{c(f)}{n}) \geq C(f) \cdot \omega^{-1}(\frac{1}{n}),$$
therefore,
$$ z \geq s \cdot \frac{c(f)}{n} \geq c(f) \cdot C(f) \frac{\omega^{-1}(\frac{1}{n})}{n},  $$
and we are done.
\end{proof}

We turn now to the following two lemmas, on which the proof of
Theorem 4 will be based.

\begin{lem} \label{L:1}
Suppose that $x_1,...,x_{n+1} \in (0,\delta)$. Then,

$$ |\log(\frac{\phi(x_{n+1})}{\phi(x_1)})| \leq
\log{\frac{n}{\omega^{-1}(\frac{1}{n})}} + C(f, \omega).$$
\end{lem}

\begin{proof}[Proof of Lemma~\ref{L:1}]
 We split the proof into 2 cases. \\ \\
Case 1: a. $x_{n+1} \in I_{x_1}$ and $\phi(x_1) < \phi(x_{n+1}).$
In this case,
$$|\log{\frac{\phi(x_1)}{\phi(x_{n+1})}}| = \log{\frac{\phi(x_{n+1})}{\phi(x_1)}} < \log{2.5},$$
due to Claim ~\ref{Cl:2} \\
b. $x_{n+1} \in I_{x_1}$ and $\phi(x_1) > \phi(x_{n+1}).$ We have
two possibilities: \\ (i) $ x_1 + \Omega(\phi(x_{n+1})) > x_{n+1}$
and
by Claim ~\ref{Cl:2} $\frac{\phi(x_1)}{\phi(x_{n+1})} < 2.5.$ \\
(ii) $ x_1 + \Omega(\phi(x_{n+1})) \leq x_{n+1}$, then :
$$n \geq a \cdot \int_{x_1}^{x_{n+1}}\frac{dt}{\phi(t)} \geq a \cdot \int_{x_{n+1}- \Omega(\phi(x_{n+1}))}^{x_{n+1}}\frac{dt}{\phi(t)}$$
$$\geq 2.5 \cdot a \cdot \frac{\Omega(\phi(x_{n+1}))}{\phi(x_{n+1})}.$$
Hence, this case is completed due to Claim ~\ref{Cl:3} \\ \\
Case 2: a. $x_{n+1} \notin I_{x_1}$ and $\phi(x_1) <
\phi(x_{n+1}).$
$$n \geq a \cdot \int_{x_1}^{x_{n+1}}\frac{dt}{\phi(t)} \geq
a \cdot \int_{x_1}^{x_1+\Omega(\phi(x_1))}\frac{dt}{\phi(t)} \geq
\frac{a}{2.5 \cdot \phi(x_1)} \cdot \Omega(\phi(x_1)),$$ in the
last inequality we have used Claim ~\ref{Cl:1}, hence we are done
due to Claim ~\ref{Cl:3}. \\b. $x_{n+1} \notin I_{x_1}$ and
$\phi(x_1)
> \phi(x_{n+1}).$
$$ n \geq \int_{x_1}^{x_{n+1}}\frac{dt}{\phi(t)} \geq a \int_{x_{n+1}-\Omega(\phi(x_1))}^{x_{n+1}}\frac{dt}{\phi(t)} \geq a \int_{x_{n+1}-\Omega(\phi(x_{n+1}))}^{x_{n+1}}\frac{dt}{\phi(t)}$$
$$\geq  \frac{a}{2.5} \cdot \frac{\Omega(\phi(x_{n+1}))}{\phi(x_{n+1})}$$
In the last inequality we have used Claim ~\ref{Cl:1}, hence we
are done due to Claim ~\ref{Cl:3}.
\end{proof}

\begin{lem} \label{L:2} Suppose that
$x_1,...,x_{n+1} \in (0,\delta).$ Then,
$$ |\log{(f^n)'(x_1)}-
\log(\frac{\phi(x_{n+1})}{\phi(x_1)})| \leq C(f) \cdot n \cdot
\omega(\frac{1}{n}).$$
\end{lem}

\begin{proof}[Proof of Lemma~\ref{L:2}]
We have
$$|\log{(f^n)'(x_1)} - \log{\frac{\phi(x_{n+1})}{\phi(x_1)}}|=
|\sum_{k=1}^{n}(\log(1+\phi'(x_k))-\log{\frac{\phi(x_{k+1})}{\phi(x_k)})}|$$
$$\leq
\sum_{k=1}^{n}|\int_{x_k}^{x_{k+1}}\frac{\phi'(t)}{\phi(t)}dt
-\log(1+\phi'(x_k))|. $$ The inequality $-\frac{y^2}{1+y} \leq
\log(1+y)-y <0 $, which is valid for all $y>-1,$ implies that
$|\log(1+y)-y| \leq \frac{y^2}{1+y}.$ In our context, we may use
both inequalities, since $\min_{x\in [0,1]}\phi'(x)>-1.$
$$\sum_{k=1}^{n}|\int_{x_k}^{x_{k+1}}\frac{\phi'(t)}{\phi(t)}dt-\log(1+\phi'(x_k))|\leq$$
$$\sum_{k=1}^{n}|\int_{x_k}^{x_{k+1}}\frac{\phi'(t)}{\phi(t)}dt-\phi'(x_k)|
+ \sum_{k=1}^{n}|\log(1+\phi'(x_k))-\phi'(x_k)|$$ $$\leq
\sum_{k=1}^{n}|\int_{x_k}^{x_{k+1}}\frac{\phi'(t)}{\phi(t)}dt-\phi'(x_k)|
+ \sum_{k=1}^{n}\frac{[\phi'(x_k)]^2}{1+\phi'(x_k)}$$
$$\leq \sum_{k=1}^{n}|\int_{x_k}^{x_{k+1}}\frac{\phi'(t)}{\phi(t)}dt-\phi'(x_k)| + \frac{1}{a} \cdot
\sum_{k=1}^{n}[\phi'(x_k)]^2.$$ Now we are going to estimate these
sums. For any $x \in [0,\delta],$ there exists $0<\theta<1$ such
that
$$\int_{x}^{x+\phi(x)}\frac{\phi'(t)}{\phi(t)}dt - \phi'(x)=\frac{\phi'(x+\theta \cdot \phi(x))}{\phi(x+\theta \cdot \phi(x))}\cdot\phi(x)-\phi'(x)=$$
$$=[\phi'(x+\theta\phi(x))-\phi'(x)]\cdot\frac{\phi(x)}{\phi(x+\theta\phi(x))} + \phi'(x)[\frac{\phi(x)}{\phi(x+\theta\phi(x))}-1].$$
By Claim ~\ref{Cl:1},
$$|[\phi'(x+\theta\phi(x))-\phi'(x)]\cdot\frac{\phi(x)}{\phi(x+\theta\phi(x))}|$$
$$\leq (|\phi'(x+\theta\phi(x))-\phi'(x)|)\cdot \max_{y \in
J_x}\frac{\phi(x)}{\phi(y)}$$ $$ \leq A \cdot \omega(\theta \cdot
\phi(x))  \leq A \cdot \omega(\phi(x)).$$ Then, there exists some
$0<\theta_1<1$, such that: $\phi(x+ \theta \phi(x))-\phi(x)= \theta
\phi(x) \cdot \phi'(x+ \theta \cdot \theta_{1} \phi(x)).$ Using it
together with Claim ~\ref{Cl:1}, we get
$$|\frac{\phi(x)}{\phi(x+\theta\phi(x))}-1|=|\frac{\phi(x)}{\phi(x)+\theta\phi(x)\phi'(x+\theta_{1}\theta\phi(x))}-1|=$$
$$= |\frac{1}{1+\theta\phi'(x+\theta_{1}\theta\phi(x))}-1| = |\frac{\theta\phi'(x+\theta_{1}\theta\phi(x))}{1+\theta\phi'(x+\theta_{1}\theta\phi(x))}|$$
$$ \leq \frac{1}{a} \cdot |\phi'(x+\theta_{1}\theta\phi(x))|
\leq \frac{3}{a} \cdot \omega(\Omega(\phi(x))).$$ Therefore
$$|\phi'(x) \cdot [\frac{\phi(x)}{\phi(x+\theta\phi(x)}-1]| \leq
\frac{3}{a} \cdot |\phi'(x)| \cdot \omega(\Omega(\phi(x)))$$
$$\leq \frac{9}{a} \cdot \omega^{2}(\Omega(\phi(x))).$$ Since $\Omega(x)
\geq x $, it follows that $\Omega(\phi(x)) \geq \phi(x).$
Additionally, $\frac{\omega(x)}{x}$ is decreasing, thus
$$\frac{\omega(\Omega(\phi(x)))}{\Omega(\phi(x))} \leq
\frac{\Omega(\phi(x))}{\phi(x)}.$$ The substitution of it yields
the following:
$$\omega^{2}(\Omega(\phi(x))) = \phi(x)
\cdot \frac{\omega(\Omega(\phi(x)))}{\Omega(\phi(x))} \leq \phi(x)
\cdot \frac{\omega(\phi(x))}{\phi(x)} =\omega(\phi(x)).$$ Adding
those results together, we have the following estimate:
$$|\int_{x}^{x+\phi(x)}\frac{\phi'(t)}{\phi(t)}dt - \phi'(x)| \leq
(A + \frac{9}{a}) \cdot \omega(\phi(x)). $$ Using the previous
estimate, we have also: $$|\phi'(x)|^2 \leq 9 \cdot
\omega^{2}(\Omega(\phi(x))) \leq 9 \cdot \omega(\phi(x)).$$ Let us
apply the above estimates for bounding our initial expressions:
$$\sum_{k=1}^{n}|\int_{x_k}^{x_{k+1}}\frac{\phi'(t)}{\phi(t)}dt-\phi'(x_k))|
+ C \cdot \sum_{k=1}^{n}[\phi'(x_k)]^2 $$ $$ \leq C(f) \cdot
\sum_{k=1}^{n}\omega(\phi(x_k)),$$ with $C(f)=10 + A + \frac{9}{a}.$
By Jensen's inequality
$$\sum_{k=1}^{n}\omega(\phi(x_k))= \sum_{k=1}^{n}\omega(x_{k+1}-x_k) \leq  n \cdot
\omega(\frac{1}{n}),$$ completing the proof of Lemma 3.
\end{proof}
Combining Lemmas 2 and 3, we get
\begin{cor} \label{C:2} Suppose that
$x_1,...,x_{n+1} \in (0,\delta).$ Then,
$$ (**) \ \ \ |\log{(f^n)'(x_1)}| \leq
\log{\frac{n}{\omega^{-1}(\frac{1}{n})}} + C(f) \cdot n \cdot
\omega(\frac{1}{n}).$$
\end{cor}

At last, we turn to the details of the proof of Theorem 4, we
shall show that estimate $(**)$ holds for each $x \in (0,1)$.
Consider the decomposition of the interval into a union of open
intervals $[0,1] \setminus $Fix$(f)$ = $\cup_{i \in I}(a_i,b_i).$
Let $x \in (0,1)$ be an arbitrary point, then $x \in (a_i,b_i)$
for some $i \in I.$ If $|b_i - a_i | \leq \delta,$ then the
proof is complete by Corollary 2. \\
There are only finitely many intervals such that $|b_i - a_i | >
\delta.$ We take one of them and divide it into 3 subintervals:
$$ [a_i,b_i] = [a_i,a_i + \delta_0 ] \cup [a_i + \delta_0 ,b_i - \delta_0]  \cup [b_i - \delta_0, b_i], $$
when $\delta_0 \leq \delta$ and $\Omega(\phi(x)) \in [a_i,b_i - \delta_0]$ for all $x \in [a_i,a_i + \delta_0].$ We denote by $n_1,n_2,n_3$ the length of the trajectory of the sequence $(x_n)$ in each of the 3 subintervals respectively.\\
It is evident that $n_2$ is bounded by some constant $N(f)$.  If
$n_3=0$ or $n_1=0$, then we are done due to Corollary 2.
Otherwise, $n = n_1 + n_2 + n_3, $

$$|\log{(f^n)'(x_1)}| \leq |\log{(f^{n_2})'(x_{n_1+1})}| + |\log{(f^{n_1})'(x_1)} +
\log{(f^{n_3})'(x_{n_1+n_2+1})}|,
$$ we continue using Lemma 2,
$$ \leq N(f) \cdot C(f) + |\log{\frac{\phi(x_{n_1}) \cdot \phi(x_n) }{\phi(x_1) \cdot \phi(x_{n_1+n_2+1})}}| + C(f)n_{1}\omega(\frac{1}{n_1}) + C(f)n_{3}\omega(\frac{1}{n_3}).$$
Note that
$$  C(f) \cdot n_{1}\omega(\frac{1}{n_1}) + C(f) \cdot n_{3}\omega(\frac{1}{n_3}) \leq 2 C(f) n \cdot \omega(\frac{1}{n}). $$
Moreover, we have the following estimate:
$$| \log{\frac{\phi(x_{n_1})}{\phi(x_{n_1+n_2+1})}}| \leq c_{i} = \max_{z \in [f^{-1}(a_{i}+\delta_0),a_{i}+\delta_0], w \in [f^{-1}(b_{i}-\delta_0),b_{i}-\delta_0] } |\log{\frac{\phi(z)}{\phi(w)}}|.$$
Now we are going to find an upper bound for $|\log{\frac{\phi(x_n)
}{\phi(x_1)}}|$. As before, we split into two cases:

a. $\phi(x_n)
> \phi(x_1).$ By using Claim 1 and the choice of $\delta_0,$ we have $$n \geq a \cdot
\int_{x_1}^{x_{n_1+n_2}}\frac{dt}{\phi(t)}$$ $$ \geq a \cdot
\int_{x_1}^{x_1+ \Omega(\phi(x_1))}\frac{dt}{\phi(t)}
 \geq \frac{a}{2.5} \cdot \frac{\Omega(\phi(x_1))}{\phi(x_1)},$$
the last inequality is due to Claim 2. Thus by Claim ~\ref{Cl:3},
we have
$$|\log{\frac{\phi(x_n) }{\phi(x_1)}}| \leq \log{\frac{1}{\phi(x_1)}} \leq C(f) \cdot
\frac{n}{\omega^{-1}(\frac{1}{n})}.$$ \\

b. $\phi(x_n) < \phi(x_1).$ Then, $$n \geq n_2 + n_3  \geq a \cdot
\int_{x_{n_1}}^{x_n}\frac{dt}{\phi(t)} $$ $$ \geq a \cdot \int_{x_n
- \Omega(\phi(x_n))}^{x_n}\frac{dt}{\phi(t)}
 \geq \frac{a}{2.5} \cdot \frac{\Omega(\phi(x_n))}{\phi(x_n)}.$$
In the same way, by Claim 3 it follows that
$$|\log{\frac{\phi(x_n) }{\phi(x_1)}}| \leq \log{\frac{1}{\phi(x_n)}} \leq C(f) \cdot
\frac{n}{\omega^{-1}(\frac{1}{n})}.$$

\end{proof}

\begin{proof}[Proof of Theorem~\ref{T:low smoothness}]
Without limiting the generality, we assume that $\omega(x)$ is a
$C^1$ smooth concave function. We assume that by $f(x) = x - \phi(x)
> 0.$ By Lemma 2,
$$\log(f^n)'(x_1) \leq \log{\frac{\phi(x_1)}{\phi(x_{n})}} + C n \omega(\frac{1}{n}).$$
Therefore, it is sufficient to show that there exists a constant $ C
> 0 $, such that for every $ n \geqslant 1 $ we have  $ \phi ( x_n)
\geqslant e^{ - C n \omega (\frac{1}{n})
 } .$ \\
 The proof is by induction. We shall determine the value of $C>0$
during the proof. Take big enough $n$, and suppose that we have $
\phi ( x_{n-1}) \geqslant e^{ - C (n-1) \omega (\frac{1}{n-1}) }
  $. We wish to prove that $ \phi ( x_n) \geqslant e^{ - C n \omega (\frac{1}{n})
 } $. Assume in a counter that $$ \phi ( x_n) < e^{ - C n \omega (\frac{1}{n})
 }. $$  Let us show that in this case we must have
$$ \phi'(t) \leqslant 3 \cdot \omega ( \frac{1}{n} ), $$ for any $ t \in
[x_n , x_{n-1} ] $.

Assume in a counter that we have $$
 \phi'(t) > 3 \omega ( \frac{1}{n} ), $$
for some  $ t \in [x_n , x_{n-1} ] .$ Note that, $$ \phi (x_{n-1}
 ) \leqslant \phi (x_{n} ) + (x_{n-1} - x_n ) \max_{ s \in
 [x_n,x_{n-1}] } \phi'(s) $$ $$\leqslant \phi (x_{n} ) + (x_{n-1} - x_n
 )\omega ( x_{n-1} - x_{n} ) = \phi (x_{n} ) + \phi (x_{n-1} )\omega ( \phi (x_{n-1}
 )),$$ hence $ \phi (x_{n-1}) \leqslant \frac{ \phi (x_{n} ) } { 1
 - \omega ( \phi (x_{n-1})) } $. Therefore, for  big enough $ n $ we
 have,  $$ \phi (x_{n-1}) \leqslant 2  \phi (x_{n} )  < 2 e^{ - C n \omega (\frac{1}{n})
 } < \frac{1}{n}, $$
the last inequality is satisfied since
$\frac{\omega(x)}{x^{\alpha}}$ is decreasing for small $x$ and $0 <
\alpha < 1.$ Thus,  for big enough $n$, we have $\phi(x_{n-1}) =
x_{n-1} -
 x_{n} < \frac{1}{n} $. We have $ \phi'(t) > 3 \omega ( \frac{1}{n} )
 $, hence $$ \phi'(x_{n}) \geqslant \phi'(t) - \omega(t - x_{n})
 \geqslant \phi'(t) - \omega(x_{n-1} - x_{n}) $$ $$\geqslant  \phi'(t) -
 \omega(\frac{1}{n}) > 2 \omega(\frac{1}{n}) .$$ In
 particular, $ \omega(x_{n}) \geqslant \phi'(x_{n}) > 2 \omega(\frac{1}{n}) > \omega(\frac{1}{n})
 $, hence $ x_n > \frac{1}{n} $. For any $ s \in [ x_n -
 \frac{1}{n} , x_n ] ,$ we have $$ \phi'(s) \geqslant \phi'(x_{n}) -
 \omega(\frac{1}{n}) >  \omega(\frac{1}{n}). $$ Therefore, by the mean value theorem $$
 \phi(x_{n}) \geqslant \phi(x_n - \frac{1}{n}) + \frac{1}{n}
 \omega(\frac{1}{n}) \geqslant \frac{1}{n}
 \omega(\frac{1}{n}) .$$ On the other hand, we have assumed that $ \phi ( x_n) < e^{ - C n \omega (\frac{1}{n})
 } $ and thus $$ \frac{1}{n} \cdot \omega(\frac{1}{n}) < e^{ - C n \omega
 (\frac{1}{n})}.$$ Now, for any $0< \alpha < 1$ and big $n$, observe
 that
$$ (C_1)^{n^{\alpha} } \cdot \omega(\frac{1}{n}) \leq n $$
where the constant $C_1$ is an increasing function of C. Now, for
any $0< \beta < 1$ and big enough $n$, we have
$$ C_2 \cdot (C_1)^{n^{\alpha}}  <   n^{1 + \beta},$$
that is a contradiction. Therefore, we have proved that $
 \phi'(t) \leqslant 3 \omega ( \frac{1}{n} ) $, for any $ t \in [x_n , x_{n-1} ]
 $.

Notice that, $ n^{\alpha} \omega( \frac{1}{n} ) \geqslant
(n-1)^{\alpha} \omega( \frac{1}{n-1} ) $. Choose some $ \alpha <
\beta < 1 $. For big enough $ n $, we have $$ (1 +
\frac{\beta}{n-1}) \omega( \frac{1}{n} ) > ( 1 +
\frac{1}{n-1})^{\alpha} \omega( \frac{1}{n} )$$ $$=
\frac{1}{(n-1)^{\alpha}} n^{\alpha} \omega(\frac{1}{n}) \geqslant
\omega( \frac{1}{n-1} ), $$ hence, $$ n \omega( \frac{1}{n} ) -
(n-1) \omega( \frac{1}{n-1} ) \geqslant ( 1 - \beta ) \omega(
\frac{1}{n} ).$$  Therefore, we conclude that $$ e^{ C n \omega(
\frac{1}{n} ) - C (n-1) \omega( \frac{1}{n-1} ) }
 > 1 + C n \omega( \frac{1}{n} ) - C (n-1) \omega( \frac{1}{n-1}
) $$ $$> 1 + (1 - \beta) C \omega( \frac{1}{n} ), $$ that is, $$ e^{
- C (n-1) \omega( \frac{1}{n-1} ) } - e^{ - C n \omega( \frac{1}{n}
) }
> ( 1- \beta) C \omega(\frac{1}{n}) e^{ - C n \omega( \frac{1}{n} )
} .$$ Finally, recall that for big $n$, $\phi(x_n) \geq 2
\phi(x_{n-1})$, and by the initial assumption that $\phi(x_n) <
e^{-C n \omega(\frac{1}{n})}$ and $\phi(x_{n-1}) \geq e^{-C (n-1)
\omega(\frac{1}{n-1})}$ we obtain

$$ ( 1- \beta) C \omega(\frac{1}{n}) e^{ - C n \omega( \frac{1}{n} )
} <  e^{ - C (n-1) \omega( \frac{1}{n-1} ) } - e^{ - C n \omega(
\frac{1}{n} ) } < \phi(x_{n-1}) - \phi(x_n) $$ $$\leqslant (x_{n-1}
- x_n) \max_{ s \in [x_n,x_{n-1}] } \phi'(s) \leqslant 3
\phi(x_{n-1}) \omega( \frac{1}{n} ) \leqslant 6 \phi(x_{n}) \omega(
\frac{1}{n} ) < $$ $$6 e^{ - C n \omega (\frac{1}{n}) } \omega(
\frac{1}{n} ) .$$ This inequality is surely false for big enough
$C$.

\end{proof}

\section{Sharpness: proofs of theorems 6 and 7} \label{S:proofs2}

\begin{proof}[Proof of Theorem~\ref{T:sharpness1}]
Define $\phi(x)=\int_{0}^{x}\omega(t)dt$ and $f(x)=x-\phi(x),$ in
some interval $[0,\varepsilon].$ Extend $f(x)$ arbitrarily
$C^{\infty}$-smoothly to the whole interval $[0,1]$ in such way
that $f(1)=f'(1)=1.$ We work in the interval $[0,\varepsilon].$ By
Lemma 3,
$$ ( \triangle) \ \ \ \log(\frac{1}{\phi(x_n)}) - C \cdot n \cdot
\omega(\frac{1}{n}) \leq \log(f^n)'(x_1).$$ Now, we estimate from
below the left hand side of $(\triangle)$.
\begin{clm} \label{Cl:9}
$$ \omega(x_n) \leq \frac{C}{n}$$
\end{clm}
\begin{proof}
We shall do it by induction. Assume that we have proved the claim
for $n-1$, namely, $$ x_{n-1} \leq \omega^{-1}(\frac{C}{n-1}).$$
Since $f$ is monotonic, it suffices to verify
$$\omega^{-1} \big(\frac{C}{n} \big) \geq f(\omega^{-1}(\frac{C}{n-1})) \geq f(x_{n-1})=x_n.$$
The last inequality is equivalent to
$$\omega^{-1}(\frac{C}{n}) + \phi(\omega^{-1}(\frac{C}{n-1})) \geq \omega^{-1}(\frac{C}{n-1}),$$
By the Mean Value theorem and the assumptions that $\omega(x)$ is
concave and monotonic and $\phi(x)$ is monotonic, it is enough to
show that
$$ \phi(\omega^{-1}(\frac{C}{n})) \geq \frac{C}{n^2 \cdot \omega'(\omega^{-1}(\frac{C}{n})},$$
denote $x=\omega^{-1}(\frac{C}{n})$ and observe
$$\phi(x) \geq \frac{\omega^{2}(x)}{C \cdot \omega'(x)}.$$
Now, recall that $\phi(x) \geq \frac{x}{4} \cdot \omega(x)$ and
consider
$$\omega'(x) \geq \frac{C}{4} \cdot \frac{\omega(x)}{x},$$
this inequality is equivalent to
$$\big( \ln{\frac{\omega(x)}{x^{\frac{4}{C}}}} \big)' \geq 0$$
this holds since we know that $\frac{\omega(x)}{x^{\alpha}}$
increases for all $0 < \alpha < 1$ on the corresponding intervals
$(0,a(\alpha)).$

\end{proof}
Recall that $\phi(x) \leq x \omega(x)$. Due to Claim 9 and the
monotonicity of $\omega^{-1},$ we have
$$\phi(x_n) \leq x_n \omega(x_n) \leq \frac{a}{n} \cdot \omega^{-1}(\frac{a}{n}).$$
Therefore,
$$ \frac{n}{a \cdot \omega^{-1}(\frac{a}{n})}\leq \frac{1}{\phi(x_n)}.$$
Substituting it into $(\triangle),$ we have

$$ \log{ \frac{n}{a \cdot \omega^{-1}(\frac{a}{n})} } - C \cdot n \cdot
\omega(\frac{1}{n}) \leq \log(f^n)'(x_1). $$ Consider any
$\varepsilon>0$, let us check that
$$ (1-\varepsilon) \cdot \log{\frac{n}{\omega^{-1}(\frac{c}{n})}} \leq \log(\frac{n}{a \cdot \omega^{-1}(\frac{a}{n})}) - C \cdot n \cdot \omega(\frac{1}{n}),$$
when $n \rightarrow \infty.$ That is equivalent to
$$ \frac{C n \omega(\frac{1}{n})}{\log{\frac{n}{a \omega^{-1}(\frac{a}{n})}}} \leq \varepsilon,$$
as $n \rightarrow \infty.$ Indeed,
$$ \frac{C n \omega(\frac{1}{n})}{\log{\frac{n}{a
\omega^{-1}(\frac{a}{n})}}} \leq \frac{C n
\omega(\frac{1}{n})}{\log{\frac{n^2}{a^2}}} \leq \frac{C}{2} \cdot
\frac{\omega(\frac{1}{n})}{\frac{\log{n}}{n}} \rightarrow 0,
$$ here we used $\omega^{-1}(x) \leq x$ and that $\lim_{x \rightarrow 0} \frac{\omega(x)}{x \log{\frac{1}{x}}}=0.$ It completes the proof of Theorem 6.
\end{proof}

\begin{proof}[Proof of Theorem~\ref{T:sharpness2}]

The proof is based on the construction presented in $[{\rm B}]$.
Let $0<\varepsilon<1$ be an arbitrary number, define
$$ \phi_{\varepsilon}(x)=x - (1+\frac{1}{x})^{-1}-x^{2+\varepsilon}
\cdot \omega(x) \cdot \sin(\frac{2 \pi}{x}) $$
$$  f_{\varepsilon}(x)= x - \phi_{\varepsilon}(x) $$
on some interval $[0,a(\varepsilon)].$ Note that $
f_{\varepsilon}(0)=0,  f'_{\varepsilon}(0)=1$ and for $0< k^{-1}
<a(\varepsilon)$, $f_{\varepsilon}(k^{-1})=(k+1)^{-1}$. It is
possible to choose $a(\varepsilon)$ in a way that \\ 1.
$f'_{\varepsilon}(x)>0$ for all $x \in [0,a(\varepsilon)].$ \\ 2.
$f_{\varepsilon}(x)$ does not admit any fixed points in $(0,a(\varepsilon)].$ \\
3. The following inequality is satisfied
$$|f'_{\varepsilon}(x)-f'_{\varepsilon}(y)| \leq C \cdot
\omega(|x-y|),$$ for all $x,y \in [0,a(\varepsilon)],$ where $C$ is an absolute constant, which does not depend on $\varepsilon.$\\
Then, for $0< k^{-1} <a(\varepsilon)$,
$$\log(f_{\varepsilon}^{N})'(k^{-1})=\sum_{j=0}^{N-1} \log{f'_{\varepsilon}(\frac{1}{k+j})}$$
$$\geq \sum_{j=0}^{N-1}\log(((k+j)^{-2}+1)^{-2} + (k+j)^{-\varepsilon} \cdot \omega((k+j)^{-1} ))$$ $$ \geq c'(\varepsilon) \cdot N \cdot (k+N-1)^{-\varepsilon} \cdot \omega((k+N-1)^{-1} ))$$
$$ \geq c(\varepsilon) \cdot N^{1-\varepsilon} \cdot \omega(N^{-1}),$$
as $N \rightarrow \infty.$ \\
We are going to construct a diffeomorphism $f \in {\rm
Diff}_{0}^{\omega}[0,1],$ which will be composed of a suitable
pasting of the frame functions $f_{\varepsilon}.$ \\
Let $\{ \varepsilon_k \}_{k \in \mathbb{N}}$ be an arbitrary
monotonically decreasing sequence of real numbers which tends to 0.
Pick two sequences $\{ a_k \}_{k \in \mathbb{N}}$, $\{ b_k \}_{k \in
\mathbb{N}}$ monotonically decreasing sequences of real numbers
which tend to 0, such that $a_k>b_{k+1},$ for all $k \in
\mathbb{N}.$
Define now \\
\[\widetilde{f}_{\varepsilon_k}(x)=
   \left\{\begin{array}{ll}
            f_{\varepsilon_k}(x-a_k), & \mbox{$x \in [a_k,a_k + a(\varepsilon_k)]$} \\
            \Psi_{k}(x), & \mbox{$x \in [a_k + a(\varepsilon_k), b_k]$} \\
          \end{array}
   \right. \]
where $\Psi_{k}(x)$ is a monotonic $C^{\infty}$-continuation of
$f_{\varepsilon_k}(x-a_k)$ to the whole interval $[a_k,b_k],$
without fixed points on the interval $[a_k + a(\varepsilon_k),
b_k]$ with the property $\Psi_{k}(b_k)=b_k, \Psi'_{k}(b_k)=1,$ and
with bounded second derivative $|\Psi_{k}''(x)|<1.$
Define \\
\[f(x)=
   \left\{\begin{array}{ll}
            \widetilde{f}_{\varepsilon_k}(x), & \mbox{$x \in [a_k,b_k]$} \\
            x, & \mbox{$x \in [0,1]\setminus \cup_{k \in \mathbb{N}}[a_k,b_k]$} \\
          \end{array}
   \right. \]

Since $\Psi_{k}(x)$ is $C^{\infty}$ with second bounded
derivative, it is not hard see that $|f'(x)-f'(y)| \leq C(f)
\omega(|x-y|),$
for all $x,y \in [0,1].$ \\
Now, choose an arbitrary $\varepsilon >0,$ there exists
$\varepsilon_k < \varepsilon.$ Pick any $m^{-1} <
a(\varepsilon_k).$ Thus, we have
$$ \log{\Gamma_{N}(f)} \geq \log{(f^N)'(a_k+m^{-1})} = \log{(f_{\varepsilon_k}^N)'(m^{-1})} $$
$$\geq c(\varepsilon_k) \cdot N^{1-\varepsilon_k} \omega(N^{-1}) \geq c(\varepsilon) \cdot N^{1-\varepsilon} \omega(N^{-1}), $$
for large $N$. Theorem 7 is proved.
\end{proof}
\textbf{Acknowledgments.} We would like to convey our gratitude
to Mikhail Sodin for his guidance, many helpful discussions and
for reading the preliminary version of this paper. We thank Leonid
Polterovich for his remarks and advises. We thank the anonymous
referee for useful remarks and simplification of the proof of
Theorem 5 and Claim 4.

\end{document}